\documentclass[12pt,leqno,twoside]{amsart}
\usepackage{amssymb,amsmath,amsthm,soul,color}

\usepackage{t1enc}
\usepackage[cp1250]{inputenc}
\usepackage{a4,indentfirst,latexsym}
\usepackage{graphics}
\usepackage{mathrsfs}
\usepackage{cite,enumitem,graphicx}
\usepackage[colorlinks=true,urlcolor=blue,
citecolor=red,linkcolor=blue,linktocpage,pdfpagelabels,
bookmarksnumbered,bookmarksopen]{hyperref}
\usepackage[english]{babel}
\usepackage[left=2.61cm,right=2.61cm,top=2.72cm,bottom=2.72cm]{geometry}
\usepackage[metapost]{mfpic}
\usepackage[normalem]{ulem}

\usepackage[colorlinks=true,urlcolor=blue,citecolor=red,linkcolor=blue,linktocpage,pdfpagelabels,bookmarksnumbered,bookmarksopen]{hyperref}

\newtheorem{Th}{Theorem}[section]

\newtheorem{Lem}[Th]{Lemma}

\newtheorem{Rem}[Th]{Remark}

\newenvironment{altproof}[1]
{\noindent%\addvspace{0.3cm}
	{\em Proof of {#1}}.}
{\nopagebreak\mbox{}\hfill $\Box$\par\addvspace{0.5cm}}

\newcommand{\wt}{\widetilde}

\newcommand{\vp}{\varphi}

\newcommand{\eps}{\varepsilon}

\def\div{\mathop{\mathrm{div}\,}}

\def\span{\mathrm{span}}

\def\N{\mathbb{N}}
\def\R{\mathbb{R}}

\def\cl{\mathrm{cl\,}}

\def\tU{{\tilde U}}

\def\V{\mathcal{V}}

\newcommand{\cC}{{\mathcal C}}

\newcommand{\cL}{{\mathcal L}}

\newcommand{\cN}{{\mathcal N}}
\newcommand{\cO}{{\mathcal O}}

\newcommand{\cS}{{\mathcal S}}

\newcommand{\cV}{{\mathcal V}}
\newcommand{\cW}{{\mathcal W}}

\newcommand{\SO}{\cS\cO}

\newcommand{\Om}{\Omega}

\def\curlop{\nabla\times}
\newcommand{\weakto}{\rightharpoonup}

\newcommand{\tu}{\widetilde{u}}
\newcommand{\tv}{\widetilde{v}}

\newcommand{\tphi}{\widetilde{\phi}}
\newcommand{\talpha}{\widetilde{\alpha}}
\newcommand{\tbeta}{\widetilde{\beta}}

\DeclareMathOperator{\dom}{dom}

\DeclareMathOperator*{\essinf}{ess\,inf}
\DeclareMathOperator*{\esssup}{ess\,sup}

\newcommand{\cnabla}{\overset{\circ}{\nabla}}

\numberwithin{equation}{section}

\allowdisplaybreaks

\begin{document}
	\title[Travelling waves for Maxwell's equations]{Travelling waves for Maxwell's equations in nonlinear and symmetric media}

\author[J. Mederski]{Jarosław Mederski}
\address[J. Mederski]{\newline\indent  	
	Institute of Mathematics,		\newline\indent 
	Polish Academy of Sciences, \newline\indent 
	ul. \'Sniadeckich 8, 00-656 Warsaw, Poland}
\email{\href{mailto:jmederski@impan.pl}{jmederski@impan.pl}}			
\author[J. Schino]{Jacopo Schino}
\address[J. Schino]{\newline\indent  	
	Faculty of Mathematics,	Informatics and Mechanics	\newline\indent
	University of Warsaw, \newline\indent
	ul. Banacha 2, 02-097 Warsaw, Poland}
\email{\href{mailto:j.schino2@uw.edu.pl}{j.schino2@uw.edu.pl}}

\date{\today}



\begin{abstract}
We look for travelling wave fields 
$$
E(x,y,z,t)= U(x,y) \cos(kz+\omega t)+ \widetilde U(x,y)\sin(kz+\omega t),\quad (x,y,z)\in\mathbb{R}^3,\, t\in\mathbb{R},
$$ 
satisfying Maxwell's equations in a nonlinear and cylindrically symmetric medium. We obtain a sequence of solutions with diverging energy  consisting of transverse magnetic field modes. In addition, we consider a   general nonlinearity, controlled by an \textit{N}-function.
\end{abstract}

\maketitle

\section*{Introduction}
\setcounter{section}{1}

We are interested in  self-trapped beams of light propagating in a nonlinear dielectric and cylindrically symmetric medium by means of solutions to the system of Maxwell’s equations
coupled with a nonlinear constitutive relation between the electric field and the electric
displacement field. More precisely, we are looking for travelling electric fields of the form
\begin{equation}\label{eq:travel_wave}
E(x,y,z,t)= U(x,y) \cos(kz+\omega t)+ \wt U(x,y)\sin(kz+\omega t),
\end{equation}
where $U,\wt U\colon\R^2\to \R^3$ are the profiles of the travelling waves, $\omega>0$ is the temporal frequency and $k\in \R\setminus\{0\}$ the spatial wave number in the direction of propagation.

In modelling optical nonlinear materials \cite{FundPhotonics}, it is natural to seek solutions having cylindrical symmetry, and in physics literature one considers usually transverse electric field modes ({\em TE modes}), or tranverse magnetic field modes ({\em TM modes}) \cite{Chen}. If the electric field is  perpendicular to the direction of propagation, i.e. 
\begin{equation}\label{eq:radialmode}
U(x,y)=\frac{\beta(r)}{r}\begin{pmatrix}
-y\\
x\\
0
\end{pmatrix},\quad  (x,y)\in\R^2,\; r=|(x,y)|,
\end{equation}
then $E=U(x,y) \cos(kz+\omega t)$ is the cylindrically symmetric TE mode ($\wt U=0$), and the system of Maxwell's equations can be reduced to the study of a nonlinear second order differential equation with boundary conditions at infinity, see Stuart \cite{Stuart91,Stuart:1993} and the references therein. If
the magnetic field $B$ is  transverse to the direction of propagation, and takes the form
\begin{equation}\label{eq:def_beta}
B=\frac{\beta(r)}{r}\begin{pmatrix}
-y\\
x\\
0
\end{pmatrix}\cos(kz+\omega t),
\end{equation}
then the equations governing these transverse magnetic field modes are more
complicated than those for TE modes because of the form of the constitutive relation in
a nonlinear optical medium. However, as observed in \cite{StuartZhou01,Chen}, when studying TM-modes, the electric field takes the form \eqref{eq:travel_wave} with two profiles $U$ and $\wt U$, with possibly $U\neq 0$ and $\wt U\neq 0$.

In order to find the exact propagation of the electromagnetic field in a nonlinear medium,
we consider a nonlinear polarization of the form
$\chi(\langle |E|^2\rangle)E,$
where $\chi$ is the scalar nonlinear susceptibility, which depends only on the time averaged intensity of $E$ over one period $T=2\pi/\omega$, that is, $\langle |E(x,y,z)|^2\rangle = \frac{1}{T}\int_0^T |E(x,y,z,t)|^2\,dt$, see \cite{Sutherland2003}.
According to Maxwell's equations, one has to 
solve the {\em nonlinear electromagnetic wave equation}
\begin{equation} \label{newe}
\nabla\times\Big(\frac{1}{\mu}\nabla \times E\Big) + \partial_{tt}\left(\varepsilon E + \chi(\langle |E|^2\rangle)E\right)=0 \quad \hbox{for }(x,y,z,t)\in\R^3\times\R,
\end{equation}
where $\eps$ is the permittivity of the medium, $\mu$ is the magnetic permeability. For simplicity, we assume that $\mu=1$; then, in absence of charges and currents, Faraday's law $\nabla\times E+\partial_t B=0$ determines the magnetic induction $B$, and Amp\'{e}re's law $\nabla\times B=\partial_t D$  together with $\div D=0$ and $\div B=0$ shows that Maxwell's equations are satisfied, where we consider the linear magnetic material law $B=\mu H=H$ and the nonlinear electric material law $D=\eps(x,y) E+\chi(\langle |E|^2\rangle)E$, see \cite{Stuart91,Stuart:1993,Agrawal,MederskiReichel}. Altogether, we find the {\em exact propagation} of the electromagnetic field in a nonlinear medium according to Maxwell's equations. Recall that in {\em cylindrically symmetric} media, $\eps$, $U$, $\wt U$ depend only on $r=|(x,y)|$.

Problem \eqref{newe} has been studied in cylindrically symmetric media in a series of previous works \cite{Stuart04,Stuart91,Stuart:1993,StuartZhou01,StuartZhou96,StuartZhou03,StuartZhou10}, where TE and TM modes have been obtained for asymptotically constant susceptibility $\chi$, i.e. in the so-called {\em saturation effect}. In this work, however, we want  to treat the probably most common type of nonlinearity in the physics and engineering literature, i.e. the {\em Kerr nonlinearity}
\begin{equation*}%\label{ex:Kerr}
\chi(\langle |E^2|\rangle) E = \chi^{(3)}\langle |E|^2\rangle E,\quad \chi^{(3)}>0.
\end{equation*}
Up to our knowledge, the first analytical work in this context is due to McLeod, Stuart and Troy  \cite{McLeodStuartTroy}, where TE modes were obtained. Namely, they showed that $\beta(r)$ must satisfy
\begin{equation}\label{eq:ODE}
\beta''+\frac{1}{r}\beta-\frac{1}{r^2}\beta+\beta^3-\beta=0,\quad \beta(0)=\beta(+\infty)=0
\end{equation}
(see also Pohl \cite{Pohl}), where $\beta$ is as in \eqref{eq:def_beta}, and for a given integer $n\geq 1$, \eqref{eq:ODE} has at least one solutions with precisely $n$ zeros in $(0,+\infty)$ whose derivative has precisely $n+1$ zeros. Then, up to some constant $\eps$ and $\chi^{(3)}$, this leads to infinitely many different TE modes $E=U(r) \cos(kz+\omega t)$ with \eqref{eq:radialmode}. Observe that a TE mode of this form is divergence free and $\wt U=0$.

The aim of this work is to find infinitely many fields $E$ of the form \eqref{eq:travel_wave} which need not be divergence free,  hence different from TE modes. Moreover, we admit a permittivity depending on $r$, $U$ need not be of the form \eqref{eq:radialmode}, and the ODE methods do not seem to be applicable. Our approach is variational, and
recall that the first author and Reichel in \cite{MederskiReichel} have recently observed that the search for travelling waves of the form \eqref{eq:travel_wave} satisfying \eqref{newe} leads to the following nonlinear elliptic problem
\begin{equation}\label{eq_m}
L \begin{pmatrix} U \\ \wt U \end{pmatrix} -\omega^2\epsilon(x,y) \begin{pmatrix} U \\ \wt U \end{pmatrix}= \omega^2\chi\left(\frac{1}{2}(|U|^2+|\wt U|^2)\right) \begin{pmatrix} U \\ \wt U \end{pmatrix},
\end{equation}
where
$$ L = \begin{pmatrix} 
-\partial_{yy}+k^2 & \partial_{xy} & 0  & 0 & 0 & k\partial_x \\
\partial_{xy} & -\partial_{xx}+k^2 & 0 & 0 & 0 & k\partial_y \\
0 & 0 & -\partial_{xx}-\partial_{yy} & k\partial_x & k \partial_y & 0 \\
0 & 0 & -k\partial_x & -\partial_{yy}+k^2 & \partial_{xy} & 0  \\
0 & 0 & -k\partial_y & \partial_{xy} & -\partial_{xx}+k^2 & 0 \\
-k\partial_x & -k\partial_y & 0 & 0 & 0 &-\partial_{xx} - \partial_{yy} 
\end{pmatrix}
$$
is the second-order differential operator $L\colon\dom(L) \subset L^2(\R^2)^6 \to L^2(\R^3)^6$, which is elliptic and self-adjoint,
and
$$
\dom(L):= \left\{\begin{pmatrix} U \\ \wt U \end{pmatrix} \in L^2(\R^2)^6: L\begin{pmatrix} U \\ \wt U \end{pmatrix} \in L^2(\R^2)^6\right\},
$$
where $L((U,\wt U)^T)$ is defined in the sense of distributions. In view of \cite[Theorem 2.3]{MederskiReichel},  the spectrum $\sigma(L)$ equals $\{0\}\cup [k^2,\infty)$, where $0$ is an eigenvalue of infinite multiplicity and $[k^2,\infty)$ consists of absolutely continuous spectrum. In \cite{MederskiReichel}, periodic media have been considered, but we borrow some ideas of the functional setting from this work to investigate the cylindrically symmetric situation.
 
We look for solutions $u\colon \R^2\to\R^6$ to the following slight generalization of \eqref{eq_m} given by
\begin{equation}\label{eq}
Lu-V(x) u=f(u)\quad\hbox{for }x\in\R^2
\end{equation}
involving the operator $L$, where we assume $f(u)=F'(u)$.  From now on, the spacial variable in $\R^2$ will be denoted by $x$ instead of $(x,y)$.
Observe that, if 
\begin{equation}\label{eq:application}
V(x)=\omega^2\eps(x),\quad F(u)=\omega^2\chi\Big(\frac{1}{2}|u|^2\Big),\quad  u\in\R^6,
\end{equation}
then \eqref{eq} leads to \eqref{eq_m} and we obtain the {\em exact} propagation of the travelling electromagnetic waves $(E,B)$, where $E$ is given by \eqref{eq:travel_wave} and $B$ is provided by Faraday's law and a subsequent time integration.

By a {\em (weak) solution} to \eqref{eq} we mean a critical point of the functional
\begin{equation}\label{eq:action}
J(u):=\frac12 b_L(u,u)
-\frac{1}{2}\int_{\R^2}V(x)|u|^2\, dx- \int_{\R^2}F(u) \, dx
\end{equation}
defined on a Banach space $X\subset L^2(\R^2)^6$ given later, where $p>2$. Here, $b_L(\cdot,\cdot)$ is the bilinear form associated with $L$, i.e. $b_L(u,\vp)= \int_{\R^2}\langle Lu,\vp\rangle\, dx$ for all $u\in \dom(L)$ and all $\vp\in \cC_0^\infty(\R^2)^6$.

Now,  for $\alpha,\talpha\in\cC_0^{\infty}(\R^2)$ and for $\beta,\tbeta\in \cC_0^\infty(\R^2)^3$
let us denote
\begin{equation*}
\begin{array}{ll} \cnabla \begin{pmatrix} \alpha\\ \talpha  \end{pmatrix}  := \begin{pmatrix} 
\partial_x \alpha \\
\partial_y \alpha \\
k \talpha \\
\partial_x \talpha \\
\partial_y \talpha \\
-k \alpha 
\end{pmatrix} \colon \R^2\to\R^6,  \qquad
\cnabla\times \begin{pmatrix} \beta\\ \tbeta  \end{pmatrix} := \begin{pmatrix}
\partial_y \beta_3- k\tbeta_2 \\
k\tbeta_1-\partial_x \beta_3 \\
\partial_x\beta_2-\partial_y \beta_1 \\
\partial_y\tbeta_3+k\beta_2 \\
-k\beta_1-\partial_x\tbeta_3\\
\partial_x\tbeta_2-\partial_y\tbeta_1
\end{pmatrix} \colon \R^2\to\R^6,
\end{array}
\end{equation*}
and observe that $L=\cnabla\times\cnabla\times$, $\cnabla\times\cnabla=0$, and thus 
$w:=\cnabla \begin{pmatrix} \alpha\\ \talpha  \end{pmatrix}\in \ker(L)$. Moreover, $J$ may be unbounded above and below and its critical points may have infinite Morse index  due to the infinite-dimensional kernel of $L$.

Due to the strongly indefinite nature of the problem and  in order to control a large class of nonlinear phenomena, we need to introduce the notion of $N$-functions. We assume that
there exists an \textit{N}-function $\Phi \colon \R \to [0,\infty)$ with the following properties.
\begin{itemize}
	\item [(N1)] $\Phi$ satisfies the $\Delta_2$ and $\nabla_2$ conditions globally\footnote{Such conditions are defined in Section \ref{sec:varsetting}.}.
	\item [(N2)] $\displaystyle \lim_{t \to 0} \frac{\Phi(t)}{t^2} = 0$.
	%\item [(N3)] For every $\alpha > 0$, $\displaystyle \lim_{t \to \infty} \frac{\Phi(t)}{e^{\alpha s^2}} = 0$.
	\item [(N3)] $\displaystyle \lim_{t \to \infty} \frac{\Phi(t)}{t^2} = \infty$.
\end{itemize}

Let us collect assumptions about $V$ and $F$:
\begin{itemize}
	\item [(V)] $V \in L^\infty(\R^2)$ is radially symmetric, positive, bounded away from $0$, and $\esssup V < k^2$.
	\item [(F0)] $F\colon\R^6\to\R$ is convex, nonnegative, and of class $\cC^1$.
	\item [(F1)] $f(u) = o(|u|)$ as $u \to 0$.
	\item [(F2)] There exist $c_1 > 0$ such that for every $u \in \R^6$
	\[
	|f(u)| \le c_1 \bigl(1 + \Phi'(|u|)\bigr) \quad \text{for all }  u \in \R^6.
	\]
%	\item [(F?)] There exists $c_2>0$ such that $\limsup_{|u| \to \infty} F(u) / \Phi(|u|) > 0$.
%	\item [(F3)] There exists $c_2>0$ such that
%	\[
%	F(u) \ge c_2 \Phi(|u|) \quad \text{for all } u \in \R^6.
%	\]
	\item [(F3)] There exists $\gamma > 2$ and $c_2>0$ such that for every $u \in \R^6$
	\[
	\frac1\gamma \langle f(u),u \rangle \ge F(u)\ge c_2 \Phi(|u|) .
	\]
%	\item [(F?)] For all $u,v \in \R^6$, $\langle f(u),u \rangle \ge 2 F(u)$ and, if $\langle f(u),v \rangle = \langle f(v),u \rangle > 0$, then
%	\[
%	F(u) - F(v) \le \frac{\langle f(u),u \rangle^2 - \langle f(u),v \rangle^2}{2 \langle f(u),u \rangle^2}.
%	\]
\end{itemize}
%\todo[inline]{Let's see if we can have $\gamma = 2$.}
Here and in the sequel, $\langle\cdot,\cdot\rangle$ denotes the Euclidean inner product in $\R^N$, $N \geq 1$. Our model example is $F(u)=\frac1p |u|^p$ with $p>2$, where $p=4$ is responsible for the Kerr effect, and then the $N$-function is of power type $|u|^p$. In general, we can admit more examples, e.g.
\begin{equation}\label{ex:Nonpower}
f(u)=\begin{cases}
|u|^{q-2}u\ln(1+|u|),\quad\hbox{ for }|u|>1,\\
|u|^{p-2}u\ln(2),\quad\hbox{ for }|u|\leq 1,
\end{cases}
\end{equation}
where $p,q>2$.  Then, $\Phi(t)=F(t)=\int_0^t f(s)\,ds$, (F1)--(F3) are satisfied, but $F$ cannot be controlled \textit{from both above and below} by any power-type $N$-function.

Let $g \in\SO(2)$ and let $u\colon \R^2\to\R^6$. We define 
$$\wt g:=\begin{pmatrix} g & 0 & 0 & 0\\
0 & 1  & 0 & 0\\
0 & 0  & g & 0\\
0 & 0  & 0 & 1\end{pmatrix}\in \R^{6\times 6}$$
and
$$(g\star u)(x) :=  \wt g u(g^{-1} x)$$ 
for a.e. $x\in\R^2$.

\begin{Th}\label{th:main} Suppose that (V), (F0)--(F3) are satisfied and $F$ is radial.
Then there exist infinitely many solutions to \eqref{eq} of the form $u_n=v_n+w_n$ such that $J(u_n)\to\infty$ as $n\to\infty$,  $v_n\in H^1(\R^2)^6$, $v_n\neq 0$, $w_n\in L^2(\R^2)^6$, $\Phi(w_n) \in L^1(\R^2)$, $Lw_n=0$, and $g\star u_n =u_n$. Moreover, each $u_n=\begin{pmatrix} U_n \\ \wt U_n \end{pmatrix} $  has the profiles of the form
\begin{equation}\label{eq:shapeofsol}
U_n = \frac{\alpha_n(x)}{|x|} \begin{pmatrix}
x_1\\
x_2\\
0
\end{pmatrix}+ \gamma_n(x) \begin{pmatrix}
0\\
0\\
1
\end{pmatrix},\quad
\wt U_n = \frac{\wt \alpha_n(x)}{|x|} \begin{pmatrix}
x_1\\
x_2\\
0
\end{pmatrix}+ \wt \gamma_n(x) \begin{pmatrix}
0\\
0\\
1
\end{pmatrix},
\end{equation}
for some radial functions $\alpha_n,\wt\alpha_n,\gamma_n,\wt\gamma_n:\R^2\to\R$.
\end{Th}

Observe that the solutions obtained in Theorem \ref{th:main}, with the first profiles of the form \eqref{eq:shapeofsol}, cannot be TE-modes, i.e. $U$ of the form \eqref{eq:radialmode} and $\wt U=0$. Therefore, we obtain a sequence of solutions with diverging energy that is different from that obtained by McLeod, Stuart, and Troy in \cite{McLeodStuartTroy}.  Moreover, if the electric field $E$ is of the form \eqref{eq:travel_wave} with profiles \eqref{eq:shapeofsol}, then  the magnetic induction
\begin{eqnarray*}
B=\curlop E=	\begin{pmatrix}\partial_{x_2}\gamma_n
-\wt \alpha_nx_2k\\
	-\partial_{x_1}\gamma_n+\wt \alpha_n x_1 k\\
		0
	\end{pmatrix}\cos(kx_3+\omega t)+
	\begin{pmatrix}\alpha_nx_2k+
	\partial_{x_2}\wt \gamma_n\\
	-\alpha_n x_1k -\partial_{x_1}\wt\gamma_n\\
	0
\end{pmatrix}\sin(kx_3+\omega t)
\end{eqnarray*}
is  transverse to the direction of propagation, hence each $u_n$ is a TM-mode. Moreover the {\em 
total electromagnetic energy} per unit interval on the $x_3$-axis is finite, i.e
 \begin{eqnarray}\label{eq:EM_Energy}
\cL(t):=\frac12\int_{\R^2}\int_{a}^{a+1}\langle E,D\rangle +\langle B,H\rangle\,dx_3\, d(x_1, x_2)<+\infty,
\end{eqnarray}
which is important in the study of self-guided beams of light in nonlinear media; see e.g. \cite{Stuart91,Stuart:1993}.

The paper is organized as follows. In Section \ref{sec:varsetting} we build the functional setting for the problem and recall some properties of the operator $L$. Our variational approach based on Cerami sequences is developed in Section \ref{sec:Cerami}. In the last Section \ref{sec:NonTE} we introduce symmetric group actions which guarantee the profiles of the form \eqref{eq:shapeofsol} and exclude TE-modes. We also prove that $\cL(t)$ is finite.

\section{Variational setting}\label{sec:varsetting}

%\subsection{Recalls about \textit{N}-functions}\label{sec:Nfunc}

We introduce the following notation. If $u=\begin{pmatrix} U\\ \tU \end{pmatrix}\in\R^6$, then $|u| := \big(|U|^2+|\tU|^2\big)^{1/2}$. In the sequel, $\langle\cdot,\cdot\rangle_2$ denotes the inner product in $L^2(\R^2)^6$ and $|\cdot|_q$ denotes the usual $L^q$-norm for $q\in [1,+\infty]$. Furthermore, we denote by $C$ a generic positive constant which may vary from one inequality to the next. We always assume that $k\neq 0$. By $\lesssim$ we denote the inequality $\leq$ up to a positive multiplicative constant.

Now, let us recall from \cite{RaoRen} some basic facts about \textit{N}-functions and Orlicz spaces. A function $\Phi \colon \R \to [0,\infty)$ is called an \textit{N}-function if and only if it is even, convex, and satisfies
\begin{equation*}
\lim_{t \to 0} \frac{\Phi(t)}{t} = 0, \quad \lim_{t \to \infty} \frac{\Phi(t)}{t} = \infty, \quad \text{and} \quad \Phi(t) > 0 \text{ for all } t > 0.
\end{equation*}
We can associate with it a second function $\Psi \colon \R \to [0,\infty)$ defined as
\[
\Psi(t) := \sup \left\{s |t| - \Phi(s) : s \ge 0\right\},
\]
which is again an \textit{N}-function and is called the \textit{complementary function} of $\Phi$. It is easy to check that $\Phi$ is the complementary function of $\Psi$.

We say that $\Phi$ satisfies the $\Delta_2$ condition globally if and only if there exists $K>0$ such that
\[
\Phi(2t) \le K \Phi(t) \quad \text{for all } t \in \R.
\]
We say it satisfies the $\nabla_2$ condition globally if and only if there exists $K'>1$ such that
\[
\Phi(K't) \ge 2 K' \Phi(t) \quad \text{for all } t \in \R.
\]
Recalling that convex functions are differentiable almost everywhere, we give the following characterization for the $\Delta_2$ and $\nabla_2$ conditions (see \cite[Theorem II.III.3]{RaoRen}).

\begin{Lem}\label{lem:cond2}
The following properties are equivalent.
\begin{itemize}
	\item $\Phi$ satisfies the $\Delta_2$ condition globally.
	\item There exists $\kappa>1$ such that $t \Phi'(t) \le \kappa \Phi(t)$ for a.e. $t \in \R$.
	\item There exists $\kappa'>1$ such that $t \Psi'(t) \ge \kappa' \Psi(t)$ for a.e. $t \in \R$.
	\item $\Psi$ satisfies the $\nabla_2$ condition globally.
\end{itemize}
\end{Lem}

The set
\[
L^\Phi := \left\{u \colon \R^2 \to \R^6 : \Phi \circ |u| \in L^1(\R^2)\right\}
\]
is a vector space if $\Phi$ satisfies the $\Delta_2$ condition globally. In this case, it is called an Orlicz space. It becomes a Banach space if endowed with the norm
\[
|u|_\Phi := \inf \left\{\alpha > 0 : \int_{\R^2} \Phi\left(\frac{|u|}{\alpha}\right) \, dx \le 1\right\},
\]
and it is reflexive if $\Phi$ satisfies the $\Delta_2$ and $\nabla_2$ conditions globally, with $L^\Psi$ begin its dual space (see \cite[Theorem IV.I.10 and Corollary IV.II.9]{RaoRen}). We can define likewise the space of scalar functions
\[
L^\Phi(\R^2) := \left\{u \colon \R^2 \to \R : \Phi \circ u \in L^1(\R^2)\right\}
\]
and the corresponding norm, still denoted $|\cdot|_\Phi$. Then, $L^\Phi = L^\Phi(\R^2)^6$ and their norms are equivalent (cf. e.g. \cite[Lemma 2.1]{MeScSz}).

An important property is the following one.

\begin{Lem}\label{lem:conj}
If $\Phi$ satisfies the $\Delta_2$ condition globally, then there exists $C>0$ such that
\[
\Psi\bigl(\Phi'(t)\bigr) \le C \Phi(t)
\]
for every $t \in \R$. In particular, if $u \in L^\Phi$, then $\Phi' \circ |u| \in L^\Psi(\R^2)$.
\end{Lem}
\begin{proof}
From Lemma \ref{lem:cond2} and \cite[Theorem I.III.3]{RaoRen}, there holds
\[
\Psi\bigl(\Phi'(t)\bigr) = t \Phi'(t) - \Phi(t) \le (\kappa - 1) \Phi(t).\qedhere
\]
\end{proof}

We conclude with the following properties about convergence and boundedness.

\begin{Lem}\label{lem:CvBdd}
Consider $(u_n) \subset L^\Phi$.
\begin{itemize}
	\item Let $u \in L^\Phi$. If $\lim_n |u_n - u|_\Phi = 0$, then $\lim_n \int_{\R^2} \Phi(u_n - u) \, dx = 0$. The opposite implication holds if $\Phi$ satisfies the $\Delta_2$ condition globally.
	\item Assume that $\Phi$ satisfies the $\Delta_2$ condition globally. Then, $(u_n)$ is bounded if and only if $\left(\int_{\R^2} \Phi(|u_n|) \, dx\right)$ is bounded.
\end{itemize}
\end{Lem}

Let us introduce the space
\begin{eqnarray*}
	\cV&:=&\Big\{u=\begin{pmatrix} U \\ \wt U \end{pmatrix} \in H^1(\R^2)^6: \Big\langle u,\cnabla \begin{pmatrix} \alpha\\ \talpha  \end{pmatrix}\Big\rangle_2=0 \mbox{ for any }\alpha,\talpha\in\cC_0^{\infty}(\R^2)\Big\}\\
	&=&\Big\{u=\begin{pmatrix} U \\ \wt U \end{pmatrix} \in H^1(\R^2)^6: \partial_{x_1} u_1 +\partial_{x_2} u_2 +k \tu_3=0, \partial_{x_1} \tu_1 +\partial_{x_2} \tu_2 -ku_3=0 \mbox{ a.e. in }\R^2\Big\},
\end{eqnarray*}
and note that it is a closed subspace of $H^1(\R^2)^6$. Let us consider the norm in $\cV$
$$\|u\|:=\Big(\sum_{i=1}^3|\nabla u_i|^2_2+k^2|u_i|^2_2+|\nabla \tu_i|^2_2+k^2|\tu_i|^2_2\Big)^{1/2},$$
which is equivalent to the standard $H^1$-norm. Let $\cW$ be the completion of 
vector fields $w=\cnabla\begin{pmatrix} \phi\\ \tphi \end{pmatrix}$, where $\phi,\tphi\in\cC_0^{\infty}(\R^2)^3$, with respect to the norm
$$\|w\|:=\big(|w|_2^2+|w|_\Phi^2\big)^{1/2}$$
so that $\cW\subset L^2(\R^2)^6\cap L^\Phi$.
We will see that $\cV\cap\cW=\{0\}$ (cf. Theorem \ref{th:Helmholtz} below), and we define a norm on the space
\begin{equation*}%\label{eq:X}
X := \cV\oplus\cW
\end{equation*}
as
$$\|v+w\|^2=\|v\|^2+\|w\|^2,\quad v\in\cV, w\in\cW.$$

If $\Phi$ is the \textit{N}-function in (F2)--(F3), then, from the $\Delta_2$ condition -- cf. (N1) -- and Lemma \ref{lem:cond2}, it has at most a power-like growth at infinity. Taking (N3) into account, there exists $p > 2$ such that
\begin{equation*}
\lim_{|t| \to \infty} \frac{|\Phi'(t)|}{|t|^{p-1}} = 0.
\end{equation*}
Then, using also (N2), for every $t\in\R$ we have
$\Phi(t)\lesssim |t|^2+|t|^p$, thus
$$
\cV\subset H^1(\R^2)^6\subset L^2(\R^2)^6 \cap L^p(\R^2)^6 \subset L^\Phi.
$$
Moreover, using Lemma \ref{lem:conj} and arguing as in \cite[Lemma 2.1]{Clement}, we can prove that $J \in \cC^1(X)$.

Next, again from the $\Delta_2$ condition and Lemma \ref{lem:cond2}, $\Phi$ has at most a power-type behaviour at the origin. Taking (N2) into account, there exists $q > 2$ such that
\[
\lim_{t \to 0} \frac{|\Phi'(t)|}{|t|^{q-1}} = \infty.
\]
Then, using also (N3) and since we can assume $q=p$, we have $|t|^p \lesssim \Phi(t)$ for every $t \in [-1,1]$ and $|t|^2 \lesssim \Phi(t)$ for every $t \in (-\infty,-1) \cup (1,\infty)$, thus
$$
L^\Phi \subset L^2(\R^2)^6 + L^p(\R^2)^6.
$$
Exploiting Lemma \ref{lem:cond2} once again, we infer
$$
L^2(\R^2)^6\cap L^{p'}(\R^2)^6\subset L^{\Psi},
$$
where $\frac{1}{p}+\frac{1}{p'}=1$. This will be used in the proof of Lemma \ref{th:Helmholtz} below.

Recall that the second-order differential operator $L\colon\dom(L) \subset L^2(\R^2)^6 \to L^2(\R^3)^6$ is elliptic and self-adjoint on the domain $\dom(L) = \{u\in L^2(\R^2)^6 : Lu\in L^2(\R^2)^6\}$. Its associated bilinear form $b_L: \dom(b_L)\times\dom(b_L)\to \R$ is given by 
$$
b_L(u,v)= \int_{\R^2} \cnabla\times u \cdot \cnabla\times v\,dx
$$
with $\dom(b_L)= \{u\in L^2(\R^2)^6 : \cnabla\times u\in L^2(\R^2)^6\}$.
Recall also that the operator $L\colon \dom(L) \subset L^2(\R^2)^6 \to L^2(\R^3)^6$ has spectrum $\sigma(L)=\{0\}\cup [k^2,\infty)$, where $0$ is an eigenvalue of infinite multiplicity and $[k^2,\infty)$ consists of absolutely continuous spectrum. Finally,
$$b_L(v,v)=\|v\|^2,\quad\hbox{for }v\in\cV.$$

\begin{Th}\label{th:Helmholtz}
	The spaces $\cV$ and $\cW$ are closed subspaces of $L^2(\R^2)^6$ and orthogonal with respect to $\langle \cdot,\cdot\rangle_2$, and $X=\cV\oplus\cW$ is the completion of  $\cC_0^{\infty}(\R^2)^6$ with respect to the norm $\|\cdot\|$. Moreover, $\cl_{L^\Phi}\cV\cap \cl_{L^\Phi}\cW=\{0\}$.
\end{Th}
\begin{proof} In order to get the Helmholtz decomposition
we argue in a similar way as in \cite[Theorem 2.1]{MederskiReichel}, and for the reader's convenience we sketch this part, which will be useful to prove also the latter statement that $\cl_{L^\Phi}\cV\cap \cl_{L^\Phi}\cW=\{0\}$.

Note that $\cV$ and $\cW$ are closed subspaces of $L^2(\R^2)^6$ and orthogonal with respect to $\langle \cdot,\cdot\rangle_2$. Let  $\vp=\begin{pmatrix} \phi\\ \tphi \end{pmatrix}\in\cC_0^{\infty}(\R^2)^6$ and let $\alpha, \talpha\in H^{1}(\R^2)\cap \cC^{\infty}(\R^2)$ be the unique solutions to 
	$$-\Delta \alpha+k^2 \alpha=-\big(\partial_{x_1} \phi_1 +\partial_{x_2} \phi_2 +k \tphi_3\big)$$
	and
	$$-\Delta \talpha+k^2 \talpha=-\big(\partial_{x_1} \tphi_1 +\partial_{x_2} \tphi_2 -k \phi_3\big)$$
	respectively. Since $\alpha$ is the Bessel potential of the $\cC_0^\infty$-function $-(\partial_{x_1} \phi_1 +\partial_{x_2} \phi_2 +k \tphi_3)$ we find that $\alpha\in W^{l,s}(\R^2)$ for every $l\in \N$ and $1<s<\infty$, cf. \cite{stein}, and similarly for $\wt \alpha$. Moreover, we find $ \alpha_n,\talpha_n\in \cC_0^{\infty}(\R^2)$ such that $\alpha_n\to\alpha$ and $\talpha_n\to\talpha$ in $W^{l,s}(\R^2)$ for any $l\in\N$ and $1<s<\infty$. In particular, this implies
	$$\vp_\cW:=\cnabla \begin{pmatrix} \alpha\\ \talpha  \end{pmatrix}=\lim_{n\to\infty}
	\cnabla \begin{pmatrix} \alpha_n\\ \talpha_n  \end{pmatrix} \in\cW,$$
	and $\vp_\cW\in W^{l,s}(\R^2)$ for all $l\in \N$ and $1<s<\infty$. In addition, $\cW$ is clearly contained in the completion of   $\cC_0^{\infty}(\R^2)^6$ with respect to the norm $\|\cdot\|$. Since, in particular, $\vp_\cW\in H^1(\R^2)^6$, we obtain by integration by parts 
	\begin{eqnarray*}
		\Big\langle \vp, \cnabla \begin{pmatrix} \beta\\ \tbeta  \end{pmatrix}\Big\rangle_2
		&=&-\int_{\R^2} (\partial_{x_1} \phi_1 +\partial_{x_2} \phi_2 +k \tphi_3)\beta+
		(\partial_{x_1} \tphi_1 +\partial_{x_2} \tphi_2 -k \phi_3)\tbeta\,dx\\
		&=&
		\int_{\R^2} (-\Delta \alpha+k^2 \alpha)\beta+
		(-\Delta \talpha+k^2 \talpha)\tbeta\,dx
		=
		\Big\langle \vp_\cW, \cnabla \begin{pmatrix} \beta\\ \tbeta  \end{pmatrix}\Big\rangle_2
	\end{eqnarray*}
	for every $\beta,\tbeta\in\cC_0^{\infty}(\R^2)$. Therefore, $\vp_\cV:=\vp-\vp_\cW\in\cV$, and
	we obtain the following Helmholtz-type decomposition
	$$\vp =\vp_\cV+\vp_\cW\quad\hbox{with }\vp_\cV\in\cV\hbox{ and }\vp_\cW\in\cW.$$
	Moreover, we have shown that $\vp\in X$. It remains to show that also $\cV$ is contained in the completion of $\cC_0^\infty(\R^2)^6$ with respect to $\|\cdot\|$. To see this, let $v\in\cV$ and take $(\vp_n)\subset \cC_0^{\infty}(\R^2)^6$
	such that $\vp_n\to v$ in $H^1(\R^2)^6$. Let us decompose $\vp_n=\vp_\cV^n+\vp_\cW^n\in\cV\oplus\cW$ and let $\cl_{L^2\cap L^\Phi}\cV$ denote the closure of $\cV$ in $L^2(\R^2)^6\cap L^\Phi$. Observe that $\big(\cl_{L^2\cap L^\Phi}\cV\big)\cap \cW=\{0\}$, so there is a continuous $L^2(\R^2)^6\cap L^\Phi$-projection of $\big(\cl_{L^2\cap L^\Phi}\cV\big)\oplus \cW$ onto $\cW$. Since $\vp_n\to v$ in $L^2(\R^2)^6\cap L^\Phi$, we infer that $\vp_\cW^n\to 0$ in $\cW$. Similarly, arguing with the closure of $\cW$ in $H^1(\R^2)^6$ we get  $\vp_\cV^n\to v$ in $\cV$. Therefore
	$$\|v-\vp_n\|^2=\|v-\vp_\cV^n\|^2+\|\vp^n_\cW\|^2\to 0$$ as $n\to\infty$, and we conclude that $\cV\oplus\cW$ is the completion of  $\cC_0^{\infty}(\R^2)^6$ with respect to the norm $\|\cdot\|$.

	Now, let $u \in \cl_{L^\Phi}\cV\cap \cl_{L^\Phi}\cW$ and $\vp=\vp_\cV+\vp_\cW$ as above.
Consider $(\vp_n) \subset \cV$ and $(\psi_n) \subset \cC_0^\infty(\R^2)$ such that $\lim_n |u - \vp_n|_\Phi = \lim_n |u - \cnabla\psi_n|_\Phi = 0$ as $n \to \infty$.
Observe that $\vp_\cV,\vp_\cW \in L^2(\R^2)^6 \cap L^{p'}(\R^2)^6 \subset L^\Psi$, whence
	\begin{eqnarray*}
		\int_{\R^2}u \vp_\cW\,dx &= \lim_{n\to\infty}
		\int_{\R^2}\vp_n \cnabla\vp\,dx=0,\\
		\int_{\R^2}u \vp_\cV\,dx &= \lim_{n\to\infty}
		\int_{\R^2}\cnabla\psi_n \vp_\cV\,dx=0.
	\end{eqnarray*}
Therefore, $\int_{\R^2}u \vp\,dx=0$ for every $\vp\in\cC_0^{\infty}(\R^2)^6$, and $u=0$.	
\end{proof}

\section{Analysis of Cerami sequences}\label{sec:Cerami}

Note that (F1) and (F2) imply that  for any $\eps>0$ there is $C_\eps>0$ such that 
\begin{equation}\label{eq:Fineq}
F(u)\leq \eps |u|^2+C_\eps \Phi(u) \quad \text{for every } u\in\R^6.
\end{equation}
For $v \in \cV$, let $w(v)$ denote the unique minimizer of
$$ w\mapsto \frac{1}{2}\int_{\R^2}V(x)|v+w|^2\, dx+\int_{\R^2}F(v+w) \, dx$$
in $\cW\subset L^2(\R^2)^6\cap L^\Phi$. Let $\wt J\colon \cV\to\R$ be given by
\begin{equation*}
\wt J(v) :=J(v+w(v))=\frac12 \big(|\nabla v|^2_2+k^2|v|^2_2\big)-\Big(\frac{1}{2}\int_{\R^2}V(x)|v+w(v)|^2\, dx+\int_{\R^2}F(v+w(v)) \, dx\Big).
\end{equation*}
In a way similar to \cite[Proof of Theorem 4.4]{BartschMederskiJFA} we show that $\widetilde{J} \in \cC^1(\cV)$ and
\[
\widetilde{J}'(v)(\vp) = J'(v + w(v))(\vp)
\]
for every $v,\vp \in \cV$. In view of (V) and \eqref{eq:Fineq}, we find $r>0$ such that
\begin{equation}\label{eq:pos}
\inf_{\|v\|=r,\, v\in \cV} \wt J(v)\geq \inf_{\|v\|=r,\, v\in \cV} J(v)>0.
\end{equation}

%\begin{Lem}\label{lem:Nehari}
%If (F?) is satisfied, then for every $u \in X$, $\psi \in \cW$, and $t \ge 0$
%\[
%J(u) \ge J\bigl(t (u + \psi)\bigr) + J'(u)\biggl(\frac{1-t^2}{2} u - t^2 \psi\biggr).
%\]
%\end{Lem}

Arguing as in the proof of \cite[Lemma 6.3]{MederskiReichel} and using Lemma \ref{lem:conj} we obtain the following property.

\begin{Lem}\label{lem:AR}
If (F3) is satisfied, then for every $\eps > 0$ there exists $c_\eps > 0$ such that for every $u = v + w \in \cV \oplus \cW$ and every $t \in [0,\sqrt{1-2/\gamma}]$
\[
J(u) \ge J(tv) + J'(u)\biggl(\frac{1-t^2}{2} u + t^2 w\biggr) - \eps t^2 |u|_2 |w|_2 - t^2 c_\eps \big|\Phi'(|u|)\big|_\Psi |w|_\Phi.
\]
\end{Lem}

%\begin{Lem}\label{lem:contP}
%Let $(X,\|\cdot\|)$ be a Banach space and $Y,Z \subset X$ closed subspaces such that $X = Y \oplus Z$. Then the projection $p_Y \colon X \to Y$ is continuous.
%\end{Lem}
%\begin{proof}
%From the closed graph theorem, it suffices to prove that
%\[
%\left\{\bigl(x,p_Y(x)\bigr) : x \in X\right\}
%\]
%is closed. Let $(x_n) \subset X$, $x \in X$, and $y \in Y$ such that $x_n \to x$ and $y_n = p_Y(x_n) \to y$ as $n \to \infty$. We want to prove that $y = p_Y(x)$. There holds $Z \ni x_n - y_n \to x - y \in Z$ as $n \to \infty$. Since $x = y + (x - y)$ and $ X = Y \oplus Z$, we have that $y = P_Y(x)$.
%\end{proof}

\begin{Lem}\label{lem:v=0}
If $v \in \cV$ is such that $\cnabla \times v = 0$ in the sense of distributions, then $v=0$.
\begin{proof}
In this proof, we shall denote points in $\R^3$ by $(x,y,z)$. Let $v = \begin{pmatrix} U \\ \wt U \end{pmatrix}$ and define $E$ via \eqref{eq:travel_wave} with $t=0$. Observe that $\cnabla \times v = 0$ implies $\nabla \times E = 0$, and $v \in \cV$ implies $\div E = 0$, both in the sense of distributions. From $\nabla \times E = 0$, \cite[Lemma 1.1 (i)]{Leinfelder}, and the fact that $E \in L^2_\textup{loc}(\R^3)$, there exists $\xi \in H^1_\textup{loc}(\R^3)$ such that $E = \nabla \xi$. Then, $0 = \div E = \Delta \xi$, so $E$ is harmonic in $\R^3$ as well. Writing $\Delta E(x,y,z) = 0$ explicitly, we obtain
\[
\begin{split}
\partial_x^2 v(x,y) \cos(kz) & + \partial_x^2 \wt v(x,y) \sin(kz) + \partial_y^2 v(x,y) \cos(kz) + \partial_y^2 \wt v(x,y) \sin(kz)\\
& - k^2 v(x,y) \cos(kz) - k^2 \wt v(x,y) \sin(kz) = 0,
\end{split}
\]
whence, taking $z=0$ and $z=\pi/(2k)$,
\[
-\Delta v + k^2 v = -\Delta \wt v + k^2 \wt v = 0 \quad \text{in } \R^2.
\]
Since $v \in H^1(\R^2)$, we conclude.
\end{proof}
\end{Lem}

The next Lemma refines the abstract condition \cite[(I7)]{MeScSz}.

\begin{Lem}\label{lem:supq}
Let $(v_n) \subset \cV$, $v \in \cV \setminus \{0\}$, $(w_n) \subset \cW$, and $(t_n) \subset (0,\infty)$ such that $t_n \to \infty$ and $v_n \weakto v$ as $n \to \infty$. Then
$$
\lim_n \frac{1}{t_n^2} \int_{\R^2} F\bigl(t_n (v_n + w_n)\bigr) \, dx = \infty.
$$
\end{Lem}
\begin{proof}
Observe that (F3) implies
\[
\frac{1}{t_n^2} \int_{\R^2} F\bigl(t_n (v_n + w_n)\bigr) \, dx \ge \frac{c_2}{t_n^2} \int_{\R^2} \Phi\bigl(t_n |v_n + w_n|\bigr) \, dx.
\]
From (N3), there exist $C,D>0$ such that for every $t \in \R$, $C \Phi(t) \ge t^2 - D$. For all $R>0$,
\[
\int_{B(0,R)} |v_n + w_n|^2 \, dx \le \frac{C}{t_n^2} \int_{\R^2} \Phi\bigl(t_n |v_n + w_n|\bigr) \, dx + \pi R^2 D,
\]
thus the statement is true if $|v_n + w_n|$ is unbounded in $L^2(B(0,R))$ for some $R>0$.\\
Now, assume that $|v_n + w_n|$ is bounded in $L^2(B(0,R))$ for all $R>0$. Up to a subsequence, $v_n(x) \to v(x)$ for a.e. $x \in \R^2$ and $w_n \to w$ in $L^2_\textup{loc}(\R^2)$ for some $w$ as $n \to \infty$. For every $\eps > 0$ and every $n$, we define
\[
\Omega_n^\eps := \left\{x \in \R^2 : |v_n(x) + w_n(x)| \ge \eps\right\}.
\]
Let us assume by contradiction that $\lim_n |\Omega_n^\eps| = 0$ for all $\eps > 0$, i.e. $\lim_n |v_n + w_n| = 0$ in measure. Then, up to a subsequence, $v_n(x) + w_n(x) \to 0$ for a.e. $x \in \R^2$ as $n \to \infty$. Consequently, $w_n \to -v$ a.e. in $\R^2$ and, therefore, $w_n \weakto -v$ in $L^2_\textup{loc}(\R^2)$ as $n \to \infty$. Since $\cnabla \times w_n = 0$ for every $n$ in the sense of distributions, this implies that the same holds for $v$. From Lemma \ref{lem:v=0}, $v=0$, a contradiction. Finally, taking $\eps>0$ such that $\lim_n |\Omega_n^\eps| > 0$ along a subsequence, we get
\[
\lim_n \int_{\R^2} \frac{\Phi\bigl(t_n |v_n + w_n|\bigr)}{t_n^2} \, dx \ge \lim_n \int_{\Om_n^\eps} \frac{\Phi\bigl(t_n |v_n + w_n|\bigr)}{t_n^2 |v_n + w_n|} |v_n + w_n| \, dx = \infty.\qedhere
\]
\end{proof}

\begin{Lem}\label{lem:neg}
If $X \subset \cV$ has a finite dimension and $(v_n) \subset X$ is such that $\lim_n \|v_n\| = \infty$, then $\lim_n \widetilde{J}(v_n) = -\infty$.
\end{Lem}
\begin{proof}
Let us denote $\tv_n := v_n / \|v_n\|$. Up to a subsequence, there exists $\tv \in X$ such that $\lim_n \tv = v$. In particular, $\tv \ne 0$. Using Lemma \ref{lem:supq}, we obtain
\[
\lim_n \frac{\widetilde{J}(v_n)}{\|v_n\|^2} \le \frac12 - \frac{1}{\|v_n\|^2} \int_{\R^2} F\bigl(\|v_n\| \tv_n + w(v_n)\bigr) \, dx = -\infty.\qedhere
\]
\end{proof}

Normally, the Ambrosetti--Rabinowitz-type condition (F3) (cf. \cite{AR}) provides the boundedness of Palais--Smale sequences by simply estimating $\wt J(v_n)-\frac{1}{\gamma}\wt J'(v_n)(v_n)$ . As we shall see below, in our problem obtaining the boundedness is much  more complicated even if $F(u)=\frac1p|u|^p$, and achieved only for Cerami sequences (cf. \cite{Cerami}).

\begin{Lem}\label{lem:bbd}
If $(v_n) \subset \cV$ is a Cerami sequence for $\widetilde J$ such that $\liminf_n \widetilde{J}(v_n) \ge 0$, then it is bounded.
\end{Lem}

We recall that a Cerami sequence for $\widetilde{J}$ is a sequence $(v_n) \subset \cV$ such that $\bigl(\widetilde{J}(v_n)\bigr)$ is bounded and $\lim_n (1 + \|v_n\|) \widetilde{J}'(v_n) = 0$, and that a Cerami sequence at the level $c \in \R$ is a Cerami sequence such that $\lim_n \widetilde{J}(v_n) = c$.

\begin{proof}[Proof of Lemma \ref{lem:bbd}]
Let $\beta \ge \limsup_n \widetilde{J}(v_n)$.
%From $F \ge 0$ and (F?), for every $\eps>0$ there exists $c_\eps>0$ such that for all $u \in \R^6$
%\begin{equation}\label{eq:Feps}
%F(u) + \eps |u|^2 \ge c_\eps \Phi(|u|).
%\end{equation}
For every $n$, let us denote $w_n := w(v_n)$ and $u_n := v_n + w_n$. From the second inequality in (F3),
\[
\frac12 \|v_n\|^2 - \frac12 \essinf V |v_n|_2^2 \ge J(u_n) + \frac12 \essinf V |w_n|_2^2 + c_2 \int_{\R^2} \Phi(|u_n|) \, dx.
\]
Adding $\|v_n\|^2 / 2 - \essinf V |v_n|_2^2 / 2$ on both sides, we obtain
\begin{multline}\label{eq:Jeps}
\|v_n\|^2 - \essinf V |v_n|_2^2\\
\ge J(u_n) + \frac12 \|v_n\|^2 - \frac12 \essinf V |v_n|_2^2 + \frac12 \essinf V |w_n|_2^2 + c_2 \int_{\R^2} \Phi(|u_n|) \, dx.
\end{multline}
Set
\[
s_n^2 := \frac12 \|v_n\|^2 - \frac12 \essinf V |v_n|_2^2 + \frac12 \essinf V |w_n|_2^2 + c_2 \int_{\R^2} \Phi(|u_n|) \, dx
\]
and $\tv_n := v_n / s_n$. %{\color{red}Up to a subsequence, there exists $\tv \in \cV$ such that $\tv_n \weakto \tv$ in $X$ and $\tv_n(x) \to \tv(x)$ for a.e. $x \in \R^2$ as $n \to \infty$.}
Assume by contradiction that $(u_n)$ is unbounded. From Theorem \ref{th:Helmholtz}, this implies that $\lim_n s_n = \infty$ along a subsequence. We want to prove that
\begin{equation}\label{eq:vnot0}
\liminf_n |\tv_n|_\Phi > 0,
\end{equation}
thus let us assume by contradiction that $|\tv_n|_\Phi \to 0$ along a subsequence as $n \to \infty$. This, together with (F1) and (F2), implies that for every $s > 0$
\begin{equation}\label{eq:Lions}
\lim_n \int_{\R^2} F(s \tv_n) \, dx = 0.
\end{equation}
%First, let us assume that (F?) holds. Let $\sigma \in (0,1)$ such that
%\begin{equation}\label{eq:sigma}
%\esssup V \le \frac\sigma2 \essinf V + (1-\sigma) k^2.
%\end{equation}
%From Lemma \ref{lem:Nehari} with $u = u_n$, $t = s/s_n$, and $\psi = w_n$, $\lim_n J'(u_n)(u_n) = 0$, $u_n \in \cM$, and \eqref{eq:Lions}, we obtain
%\begin{equation*}
%\begin{split}
%\beta & \ge \liminf_n J(u_n)\\
%& \ge \liminf_n J(s \tv_n) + \liminf_n J'(u_n)\biggl(\frac{1-(s/s_n)^2}{2} u_n + (s/s_n)^2 w_n\biggl)\\
%& =  \liminf_n J(s \tv_n) \ge \frac{s^2}{2} \liminf_n \bigl(\|\tv_n\|^2 - \esssup V |\tv_n|_2^2\bigr).
%\end{split}
%\end{equation*}
%In addition, \eqref{eq:sigma} yields
%\begin{equation*}
%\begin{split}
%\|\tv_n\|^2 - \esssup V |\tv_n|_2^2 & \ge \|\tv_n\|^2 - (1-\gamma) k^2 |\tv_n|_2^2 - \frac\gamma2 \essinf V |\tv_n|_2^2\\
%& \ge \gamma \biggl(\|\tv_n\|^2 - \frac12 \essinf V |\tv_n|_2^2\biggr) \ge 1,
%\end{split}
%\end{equation*}
%where the last inequality follows from \eqref{eq:Jeps}. In conclusion, taking $s = 2 \sqrt{\beta / \sigma}$, we get the contradiction
%\[
%\beta \ge \frac{s^2}{2} \gamma = 2\beta.
%\]
%Now, let us assume that (F?) and (F?) hold.
From (F3), there holds
\[
J(u_n) - \frac12 J'(u_n)(u_n) \ge \frac{\gamma-2}{2} \int_{\R^2} F(u_n) \, dx,
\]
thus $(u_n)$ is bounded in $L^\Phi$. From Theorem \ref{th:Helmholtz}, $(w_n)$ is bounded in $L^\Phi$. Moreover, there exists $\delta > 0$ such that $s_n \ge \delta |u_n|_2^2 \ge \delta |w_n|_2^2$. From Lemma \ref{lem:AR} with $u = u_n$ and $t = s/s_n$ (observe that $t \le \sqrt{1-2/\gamma}$ for $n \gg 1$) and Lemma \ref{lem:conj},
for every $\eps>0$ there exists $c_\eps>0$ such that
\begin{equation*}
\begin{split}
\beta & \ge \liminf_n J(u_n)\\
& \ge \liminf_n J(s \tv_n) + \liminf_n J'(u_n) \biggl(\frac{1-(s/s_n)^2}{2} u_n + (s/s_n)^2 w_n\biggr)\\
& \quad - \eps \limsup_n (s/s_n)^2 |u_n|_2 |w_n|_2 - c_\eps \limsup_n (s/s_n)^2 \big|\Phi'(|u_n|)\big|_\Psi |w_n|_\Phi\\
& \ge \liminf_n J(s \tv_n) + \liminf_n J'(u_n) \biggl(\frac{1-(s/s_n)^2}{2} u_n + (s/s_n)^2 w_n\biggr) - s^2 \frac{\varepsilon}{\delta}.
\end{split}
\end{equation*}
Observe that the minimality of $w_n$ implies $J'(u_n)(w_n) = 0$, hence, since $\lim_n J'(u_n)(u_n) = 0$, we obtain
\begin{equation*}
\beta \ge \liminf_n J(u_n) - s^2 \frac{\varepsilon}{\delta},
\end{equation*}
which, together with \eqref{eq:Lions}, yields
\begin{equation*}
\beta \ge s^2 \biggl(\liminf_n\bigl(\|\tv_n\|^2 - \esssup V |\tv_n|_2^2\bigr) - \frac{\eps}{\delta^2}\biggr).
\end{equation*}
%{\color{red}$\lim_n J'(u_n)(u_n) = 0$, the minimality of $w_n$, and \eqref{eq:Lions}, for every $\eps>0$ there exists $c_\eps>0$ such that
%\begin{equation*}
%\begin{split}
%\beta & \ge \liminf_n J(u_n)\\
%& \ge \liminf_n J(s \tv_n) + \liminf_n J'(u_n) \biggl(\frac{1-(s/s_n)^2}{2} u_n + (s/s_n)^2 w_n\biggr)\\
%& \quad - \eps \limsup_n (s/s_n)^2 |u_n|_2 |w_n|_2 - c_\eps \limsup_n (s/s_n)^2 \big|\Phi'(|u_n|)\big|_\Psi |w_n|_\Phi\\
%& \ge \liminf_n J(s \tv_n) - s^2 \frac{\eps}{\delta^2} = s^2 \biggl(\liminf_n\bigl(\|\tv_n\|^2 - \esssup V |\tv_n|_2^2\bigr) - \frac{\eps}{\delta^2}\biggr).
%\end{split}
%\end{equation*}}
Taking $\sigma \in (0,1)$ such that
\begin{equation*}
\esssup V \le \frac\sigma2 \essinf V + (1-\sigma) k^2,
\end{equation*}
we obtain, for every $n$,
\begin{equation*}
\begin{split}
\|\tv_n\|^2 - \esssup V |\tv_n|_2^2 & \ge \|\tv_n\|^2 - (1-\sigma) k^2 |\tv_n|_2^2 - \frac\sigma2 \essinf V |\tv_n|_2^2\\
& \ge \sigma \biggl(\|\tv_n\|^2 - \frac12 \essinf V |\tv_n|_2^2\biggr) \ge \sigma + o_n(1),
\end{split}
\end{equation*}
where the last inequality comes from \eqref{eq:Jeps} and $J(u_n) \ge o_n(1)$. In conclusion, taking
$s = 2\sqrt{\beta / \sigma}$ and $\eps = \sigma \delta^2 / 8$, we get the contradiction
\[
\beta \ge s^2 \biggl(\frac{\sigma}{2} - \frac{\eps}{\delta^2}\biggr) = \frac32 \beta.
\]
Using \eqref{eq:vnot0} and the fact that $(\tv_n)$ is bounded, we find $\tv \in \cV \setminus \{0\}$ such that $\tv_n \weakto \tv$ as $n \to \infty$ along a subsequence. Then, from Lemma \ref{lem:supq},
\[
0 = \lim_n \frac{J(u_n)}{s_n^2} \le \lim_n \frac12 \|\tv_n\|^2 - \frac{1}{s_n^2} \int_{\R^2} F(s_n \tv_n + w_n) \, dx = -\infty,
\]
a contradiction.
\end{proof}

Now we introduce the cylindrical symmetry of the problem. Let 
$$g=\begin{pmatrix} \cos \alpha & -\sin  \alpha\\
\sin  \alpha &  \cos \alpha\end{pmatrix} \in\SO(2)$$ with $\alpha \in\R$, and let $U\colon\R^2\to\R^3$. We define 
$$(g\star U)(x) := \begin{pmatrix} g & 0\\
0 & 1\end{pmatrix} U(g^{-1} x)
=\begin{pmatrix}  U_1 (g^{-1}x) \cos\alpha - U_2 (g^{-1}x) \sin\alpha\\
U_1 (g^{-1}x) \sin\alpha + U_2 (g^{-1}x) \cos\alpha\\
U_3 (g^{-1}x)\end{pmatrix}$$ 
for $x\in\R^2$, where $g^{-1}x=(x_1\cos \alpha + x_2\sin  \alpha,-x_1\sin  \alpha +x_2\cos \alpha)$.

\begin{Lem} 
	There holds
	$$\Big|\cnabla\times \begin{pmatrix} (g\star U)(x)\\ (g\star \wt U)(x) 
	\end{pmatrix}\Big| =\Big|\cnabla\times \begin{pmatrix} U (g^{-1}x)\\ \wt U(g^{-1}x) 
	\end{pmatrix} \Big|$$
	for  $U,\wt U\colon\R^2\to\R^3$ and $g\in \SO(2)$.
\end{Lem}

\begin{proof}
	In the following, when we write $\partial_{x_i} U_j(g^{-1}x)$ we always mean $(\partial_{x_i} U_j)(g^{-1}x)$.\\
	Observe that
	\begin{eqnarray*}
		\cnabla\times \begin{pmatrix} g\star U\\ g\star \wt U  \end{pmatrix} &:=& \begin{pmatrix}
			\partial_{x_1} U_3 (g^{-1}x)\sin\alpha+\partial_{x_2} U_3 (g^{-1}x)\cos\alpha- k\wt U_1 (g^{-1}x) \sin\alpha - k\wt U_2 (g^{-1}x) \cos\alpha\\
			k \wt U_1 (g^{-1}x) \cos\alpha - k \wt U_2 (g^{-1}x) \sin\alpha-\partial_{x_1} U_3(g^{-1}x)\cos\alpha+\partial_{x_2} U_3(g^{-1}x) \sin\alpha\\
			\partial_{x_1}U_2 (g^{-1}x)-\partial_{x_2}U_1 (g^{-1}x) \\
			\partial_{x_1} \wt U_3 (g^{-1}x)\sin\alpha+\partial_{x_2} \wt U_3 (g^{-1}x)\cos\alpha+k U_1 (g^{-1}x) \sin\alpha +k U_2 (g^{-1}x) \cos\alpha\\
			-k U_1 (g^{-1}x) \cos\alpha +k U_2 (g^{-1}x) \sin\alpha-\partial_{x_1}\wt U_3(g^{-1}x)\cos\alpha+\partial_{x_2} \wt U_3(g^{-1}x) \sin\alpha\\
			\partial_{x_1}\wt U_2(g^{-1}x)-\partial_{x_2} \wt U_1(g^{-1}x)
		\end{pmatrix}\\
		&=&
		\begin{pmatrix}
			\partial_{x_1} U_3 (g^{-1}x)\sin\alpha+\partial_{x_2} U_3 (g^{-1}x)\cos\alpha- k\wt U_1 (g^{-1}x) \sin\alpha - k\wt U_2 (g^{-1}x) \cos\alpha\\-\partial_{x_1} U_3(g^{-1}x)\cos\alpha+\partial_{x_2} U_3(g^{-1}x) \sin\alpha+
			k \wt U_1 (g^{-1}x) \cos\alpha - k \wt U_2 (g^{-1}x) \sin\alpha\\
			\partial_{x_1}U_2 (g^{-1}x)-\partial_{x_2}U_1 (g^{-1}x) \\
			\partial_{x_1} \wt U_3 (g^{-1}x)\sin\alpha+\partial_{x_2} \wt U_3 (g^{-1}x)\cos\alpha+k U_1 (g^{-1}x) \sin\alpha +k U_2 (g^{-1}x) \cos\alpha\\
			-\partial_{x_1}\wt U_3(g^{-1}x)\cos\alpha+\partial_{x_2} \wt U_3(g^{-1}x) \sin\alpha
			-k U_1 (g^{-1}x) \cos\alpha +k U_2 (g^{-1}x) \sin\alpha\\
			\partial_{x_1}\wt U_2(g^{-1}x)-\partial_{x_2} \wt U_1(g^{-1}x)
		\end{pmatrix}
	\end{eqnarray*}
	and we easily conclude that
	\[
	\Big|\cnabla\times \begin{pmatrix} (g\star U)(x)\\ (g\star \wt U)(x) 
	\end{pmatrix}\Big|^2=\Big|\cnabla\times \begin{pmatrix} U (g^{-1}x)\\ \wt U(g^{-1}x) 
	\end{pmatrix} \Big|^2. \qedhere
	\]
\end{proof}

Let 
\begin{eqnarray*}
\cV_{\cS\cO(2)}&:=& \big\{v\in\cV:\; g\star v=v\hbox{ for any }g\in\cS\cO(2)\big\},\\
\cW_{\cS\cO(2)}&:=& \big\{w\in\cW:\; g\star w=w\hbox{ for any }g\in\cS\cO(2)\big\}.
\end{eqnarray*}
From now on, for $v \in \cV_{\cS\cO(2)}$, $w(v)$ is the unique minimizer of
$$
w\mapsto \frac{1}{2}\int_{\R^2} V(x)|v+w|^2\, dx+\int_{\R^2}F(v+w) \, dx
$$
in $\cW_{\cS\cO(2)}$.

\begin{Lem}\label{lem:Cer}
The functional $\widetilde J|_{\cV_{\cS\cO(2)}}$ %$\widetilde J|_{\cV_{\cS\cO(2)} \oplus \cW_{\cS\cO(2)}}$
satisfies the Cerami condition at every nonngative level (that is, every Cerami sequence at a nonnegative level has a converging subsequence).
\end{Lem}
\begin{proof}
First, we claim that $\cV_{\cS\cO(2)}$ is compactly embedded into $L^r(\R^2)^6$ for any $r>2$.
Indeed, let $(v_n)\subset \cV_{\cS\cO(2)}$ be such that $v_n\weakto 0$ as $n \to \infty$. Passing to a subsequence $v_n\to 0$ a.e. on $\R^2$ as $n \to \infty$. Then, $(|v_n|)$ is bounded in $H^1(\R^2)$ and $\cO(2)$-invariant. Hence, $|v_n|\to v_0$ in $L^r(\R^2)$ along a subsequence as $n \to \infty$, which yields $v_0=0$.\\
Next, let $(v_n) \subset \cV_{\cS\cO(2)}$ be a Cerami sequence for $\widetilde J|_{\cV_{\cS\cO(2)}}$ %$\widetilde J|_{\cS\cO(2)}$
as in the statement. From Lemma \ref{lem:bbd}, there exists $v \in \cV_{\cS\cO(2)}$ such that $v_n \weakto v$ along a subsequence as $n \to \infty$. In particular, up to taking other subsequences, $v_n \to v$ in $L^\Phi$ and a.e. in $\R^2$ as $n \to \infty$. From the continuity of $w \colon L^\Phi_{\cS\cO(2)} \to \cW_{\cS\cO(2)}$, which is proved as in \cite[Lemma 4.6]{MeScSz}, $(w_n) \subset \cW$ is bounded, where $w_n := w(v_n)$.
%$\lim_n w(v_n) = w(v)$ in $L^\Phi$.
There holds
\begin{equation*}
\begin{split}
o_n(1) & = \widetilde{J}'(v_n)(v_n) - \widetilde{J}'(v_n)(v)\\
& = b_L(v_n,v_n - v) - \int_{\R^2} V(x) \langle v_n + w_n,v_n - v \rangle + \langle f(v_n + w_n),v_n - v \rangle \, dx\\
& = b_L(v_n,v_n - v) - \int_{\R^2} V(x) \langle v_n,v_n - v \rangle \, dx + o_n(1),
\end{split}
\end{equation*}
which implies that $v_n \to v$ as $n \to \infty$.
\end{proof}

\begin{Rem}
Our proof can be also adopted in the situation where the Ambrosetti--Rabinowitz-type condition (F3) is not satisfied, e.g. \eqref{ex:Nonpower} with $p>q=2$. Then one needs to consider a monotonicity-type condition as in \cite{BartschMederski,MederskiENZ,MederskiReichel} given as follows:
\begin{itemize}
	\item[(F4)] If $ \langle f(u),v\rangle = \langle f(v),u\rangle >0$, then
	$\ \displaystyle F(u) - F(v)
	\le \frac{\langle f(u),u\rangle^2-\langle f(x,u),v\rangle^2}{2\langle f(u),u\rangle}$.\\
	Moreover, $\langle f(u),u\rangle\geq 2F(u)$ for every  $u\in\R^2$.
\end{itemize}
If $F$ satisfies (F1), (F2), (F4) and $F(u) \gtrsim \Phi(|u|)$ for every $u\in\R^2$, then one can work with the Nehari--Pankov manifold for $J$, or Nehari manifold for $\wt J$:
$$\cN:=\{u\in (\cV\oplus\cW)\setminus\cW: J'(u)|_{\R u\oplus \cW}=0\}=\{u\in \V\setminus\{0\}: \wt J'(u)(u)=0\},$$ and build a critical point theory on this constraint as in \cite{BartschMederski,BartschMederskiJFA}. We leave the details to the reader.
\end{Rem}

\section{Non-TE modes}\label{sec:NonTE}
In this section, we will work with the space
\[
X_{\SO(2)} := \left\{u \in X : u = g \star u \text{ for all } g \in \SO(2)\right\}.
\]

For any $\SO(2)$-equivariant profile $U\colon\R^2\to\R^3$, there is a unique decomposition 
$$U=U_\rho+U_\tau+U_\zeta,$$ such that $U_\rho$, $U_\tau$, $U_\zeta$ are vector fields given by
\begin{equation*}
U_\rho(x) = \frac{\alpha_\rho(x)}{|x|} \begin{pmatrix}
x_1\\
x_2\\
0
\end{pmatrix}, \quad
U_\tau(x) = \frac{\alpha_\tau(x)}{|x|} \begin{pmatrix}
-x_2\\
x_1\\
0
\end{pmatrix}, \quad
U_\zeta(x) = \alpha_\zeta(x) \begin{pmatrix}
0\\
0\\
1
\end{pmatrix},
\end{equation*}
for some  $\SO(2)$-invariant  functions $\alpha_\rho,\alpha_\tau,\alpha_\zeta \colon \R^2 \to \R$, cf.  \cite[Lemma 1]{ABDF}. In fact, $\alpha_\rho(x) = \bigl(U_1(x)x_1+U_2(x)x_2\bigr) |x|^{-1}$, $\alpha_\tau(x) = \bigl(-U_1(x)x_2+U_2(x)x_1\bigr) |x|^{-1}$, $\alpha_\zeta(x) = U_3(x)$ for a.e. $x\in\R^2$.

Now for any $u=\begin{pmatrix} U \\ \wt U \end{pmatrix}\in X_{\SO(2)}$ we define
\begin{equation*}
u_\rho(x) = \begin{pmatrix} U_\rho \\ \wt U \end{pmatrix}, \quad
u_\tau(x) = \begin{pmatrix} U_\tau \\ \mathbf{0} \end{pmatrix}, \quad
u_\zeta(x) = \begin{pmatrix} U_\zeta \\ \mathbf{0} \end{pmatrix},
\end{equation*}
where $\mathbf{0}=\begin{pmatrix} 0\\ 0\\0 \end{pmatrix}.$

\begin{Lem}\label{le:ABDF}
For every $u \in X_{\SO(2)}$ we have $u_\tau, u_\rho+u_\zeta \in X_{\SO(2)}$ and $u = u_\rho + u_\tau + u_\zeta$, and, for a.e. $x \in \R^2$, $u_\rho(x)$, $u_\tau(x)$, and $u_\zeta(x)$ are mutually orthogonal in $\R^6$. Moreover, if $u\in \cV_{\SO(2)}$, then $u_\tau, u_\rho + u_\zeta \in \cV_{\SO(2)}$. If $u\in \cW_{\SO(2)}$, then $u_\tau=0$ and $u_\rho + u_\zeta \in \cW_{\SO(2)}$. Finally,
\begin{equation}\label{eq:curlsymm}
\langle\cnabla\times u_\rho (x), \cnabla\times u_\tau(x)\rangle=\langle\cnabla\times u_\tau (x), \cnabla\times u_\zeta(x)\rangle=0
\end{equation}
for a.e. $x\in\R^2$.
\end{Lem}
\begin{proof}
Since $U$ and $\wt U$ are $\SO$-equivariant, %$\alpha_\rho(x)=U_1(x)x_1+U_2(x)x_2$, $\alpha_\tau(x)=-U_1(x)x_2+U_2(x)x_1$, $\alpha_\zeta(x)=U_3(x)$
$\alpha_\rho$, $\alpha_\tau$, and $\alpha_\zeta$ are $\SO$-invariant, hence radial. Since $u_\rho(x)$, $u_\tau(x)$, and $u_\zeta(x)$ are mutually orthogonal in $\R^6$, then
$u_\rho,u_\tau,u_\zeta \in L^\Phi_{\SO(2)}$.  
Observe that
$$u_\rho(x) = \begin{pmatrix}
\alpha(x) x_1\\
\alpha(x) x_2\\
0\\
\wt\alpha_1(x)\\
\wt\alpha_2(x)\\
\wt\alpha_3(x)
\end{pmatrix},\quad
u_\tau(x) = \begin{pmatrix}
-\beta(x) x_2\\
\beta(x) x_1\\
0\\
0\\
0\\
0
\end{pmatrix},\quad u_\zeta(x) =
\begin{pmatrix}
0\\
0\\
\gamma(x)\\
0\\
0\\
0
\end{pmatrix}$$
for some $\wt \alpha_1, \wt \alpha_2$, and radial $\alpha,\beta,\gamma,\wt \alpha_3$.
Note that
$$\cnabla\times u_\rho(x) = \begin{pmatrix}
-k \wt\alpha_2\\
k\wt\alpha_1\\
\partial_{x_1}\alpha x_2-\partial_{x_2}\alpha x_1\\
\partial_{x_2}\wt\alpha_3 + k \alpha x_2\\
-k\alpha x_1-\partial_{x_1}\wt\alpha_3\\
\partial_{x_1}\wt \alpha_2-\partial_{x_2}\wt \alpha_1
\end{pmatrix}=
 \begin{pmatrix}
-k \wt\alpha_2\\
k\wt\alpha_1\\
0\\
\partial_{x_2}\wt\alpha_3 + k \alpha x_2\\
-k\alpha x_1-\partial_{x_1}\wt\alpha_3\\
\partial_{x_1}\wt \alpha_2-\partial_{x_2}\wt \alpha_1
\end{pmatrix},\quad
\cnabla\times u_\tau(x) = \begin{pmatrix}
0\\
0\\
\partial_{x_1}(\beta x_1)+\partial_{x_2}(\beta x_2)\\
k \beta x_1\\
k\beta x_2\\
0
\end{pmatrix},$$
and since
$$(\partial_{x_2}\wt\alpha_3 + k \alpha x_2)x_1+(-k\alpha x_1-\partial_{x_1}\wt\alpha_3)x_2=\partial_{x_2}\wt\alpha_3x_1-\partial_{x_1}\wt\alpha_3x_3=0,$$
then 
$$\langle\cnabla\times u_\rho (x), \cnabla\times u_\tau(x)\rangle=0.$$
Moreover
$$\cnabla\times u_\zeta(x) = \begin{pmatrix}
\partial_{x_2}\gamma\\
-\partial_{x_1}\gamma\\
0\\
0\\
0\\
0
\end{pmatrix}$$
and
$$\langle\cnabla\times u_\tau (x), \cnabla\times u_\zeta(x)\rangle=0.$$
Therefore, \eqref{eq:curlsymm} holds and 
$$|\cnabla\times u|^2=|\cnabla\times u_\tau|^2+|\cnabla\times (u_\rho+u_\zeta)|^2.$$
Hence, if  $u\in \cV_{\SO(2)}$, then $u_\tau, u_\rho + u_\zeta \in \cV_{\SO(2)}$. If $u\in \cW_{\SO(2)}$, then $u_\tau=0$ and $u_\rho + u_\zeta \in \cW_{\SO(2)}$. In view of the Helmholtz decomposition Theorem \ref{th:Helmholtz}, we see that $u_\tau \in X_{\SO(2)}$, hence $u_\rho + u_\zeta = u - u_\tau \in X_{\SO(2)}$.
\end{proof}

We consider the action $\cS \colon X_{\SO(2)} \to X_{\SO(2)}$ given by
\begin{equation*}
\cS u := u_\rho - u_\tau + u_\zeta
\end{equation*}
and note that \eqref{eq:curlsymm}
yields
\begin{equation*}
\int_{\R^2} |\cnabla \times (\cS u)|^2 \, dx = \int_{\R^2} |\cnabla \times u_\tau|^2 \, dx + \int_{\R^2} |\cnabla \times (u_\rho + u_\zeta)|^2 \, dx = \int_{\R^2} |\cnabla \times u|^2 \, dx.
\end{equation*}
Hence we obtain that  $\cS$ in an isometry and $J|_{X_{\SO(2)}}$ is also invariant under $\cS$. Let
$$\big(X_{\SO(2)}\big)^\cS:=\big\{u\in X_{\SO(2)}:\cS u =u\big\}=\big\{u\in X_{\SO(2)}: u =u_\rho+u_\zeta\big\},$$
and in view of  the Palais principle of symmetric criticality, critical points of $J|_{(X_{\SO(2)})^\cS}$ are critical points of the free functional $J$.

In a similar way we decompose the profile $\wt U\colon\R^2\to\R^3$
$$\wt U=\wt U_\rho+\wt U_\tau+\wt U_\zeta,$$ where $\wt U_\rho$, $\wt U_\tau$, $\wt  U_\zeta$ are vector fields given by
\begin{equation*}
\wt  U_\rho(x) = \frac{\wt  \alpha_\rho(x)}{|x|} \begin{pmatrix}
x_1\\
x_2\\
0
\end{pmatrix}, \quad
\wt  U_\tau(x) = \frac{\wt \alpha_\tau(x)}{|x|} \begin{pmatrix}
-x_2\\
x_1\\
0
\end{pmatrix}, \quad
\wt U_\zeta(x) = \wt \alpha_\zeta(x) \begin{pmatrix}
0\\
0\\
1
\end{pmatrix},
\end{equation*}
for some  $\SO(2)$-invariant  functions $\wt  \alpha_\rho,\wt  \alpha_\tau,\wt  \alpha_\zeta \colon \R^2 \to \R$. Arguing as in Lemma \ref{le:ABDF} we get
$$\langle\cnabla\times\wt  u_\rho (x), \cnabla\times \wt u_\tau(x)\rangle=\langle\cnabla\times\wt  u_\tau (x), \cnabla\times \wt u_\zeta(x)\rangle=0$$
for a.e. $x\in\R^2$, where
\begin{equation*}
\wt u_\rho(x) = \begin{pmatrix} U \\ \wt U_\rho \end{pmatrix}, \quad
\wt u_\tau(x) = \begin{pmatrix}  \mathbf{0}\\ \wt U_\tau  \end{pmatrix}, \quad
\wt u_\zeta(x) = \begin{pmatrix}  \mathbf{0}\\ \wt U_\zeta  \end{pmatrix}.
\end{equation*}
Therefore we may define  the action $\wt \cS \colon \big(X_{\SO(2)}\big)^\cS\to \big(X_{\SO(2)}\big)^\cS$ given by
\begin{equation*}
\wt\cS u := \wt u_\rho - \wt u_\tau + \wt u_\zeta
\end{equation*}
and observe that
$$\Big(\big(X_{\SO(2)}\big)^\cS\Big)^{\wt \cS}=\Big\{u= \begin{pmatrix} U \\ \wt U \end{pmatrix}\in X_{\SO(2)}: U =U_\rho+U_\zeta,\; \wt U =\wt U_\rho+\wt U_\zeta\Big\}.$$

\begin{altproof}{Theorem \ref{th:main}}
In virtue of Lemma \ref{le:ABDF}, we consider the decomposition
\begin{equation*}
\Big(\big( X_{\SO(2)} \big)^\cS\Big)^{\wt \cS} = \Big(\big( \cV_{\SO(2)} \big)^\cS\Big)^{\wt \cS} \oplus \Big(\big( \cW_{\SO(2)} \big)^\cS\Big)^{\wt \cS}.
\end{equation*}
Since $\Big(\big( \cV_{\SO(2)} \big)^\cS\Big)^{\wt \cS}$ is an infinite-dimensional separable Hilbert space, we can consider an orthonormal basis $(v_n)_{n=1}^\infty \subset \Big(\big( \cV_{\SO(2)} \big)^\cS\Big)^{\wt \cS}$. Next, for every $n \ge 1$, we define
\begin{equation*}
Y_n := \bigoplus_{j=1}^n \span \{v_j\} \quad \text{and} \quad Z_n := \overline{\bigoplus_{j=n}^\infty \span \{v_j\}}.
\end{equation*}
In view of Lemma \ref{lem:Cer}, we can directly apply the Fountain Theorem with a Cerami condition \cite[Theorem 2.9]{Liu2010}, cf. \cite{Bartsch}, and conclude by the Palais principle of symmetric criticality. In particular, we need to prove the existence of two sequences $(r_n)$ and $(\rho_n)$ such that $\rho_n > r_n > 0$ for every $n \ge 1$ and the following properties hold:
\begin{itemize}
	\item [(i)] $\displaystyle \lim_n \left(\inf\left\{ \wt J(v) : v \in Z_n \text{ and } \|v_n\| = r_n \right\}\right) = \infty$,
	\item [(ii)] For all $n \ge 1$, $\displaystyle \max\left\{ \wt J(v) : v \in Y_n \text{ and } \|v\| = \rho_n \right\} \le 0$.
\end{itemize}
We check (i) first. Fix $\varepsilon \in \left(0,\frac12 k^2 - \frac12 \esssup V\right)$ and let $C_\varepsilon > 0$ such that \eqref{eq:Fineq} holds. Then, denoting $c_\varepsilon := \frac12 k^2 - \frac12 \esssup V - \varepsilon > 0$, from \eqref{eq:pos}, for every $v \in \cV$ there holds
\begin{equation*}
\wt J(v) \ge c_\varepsilon \|v\|^2 - C_\varepsilon \int_{\R^2} \Phi(|v|) \, dx.
\end{equation*}
Let $\delta \in (0,c_\varepsilon / C_\varepsilon)$. From (N2), there exists $r > 0$ such that $\Phi(t) \le \delta t^2$ for all $|t| \le r$. Now, let $\kappa > 1$ be the constant appearing in the characterization of the $\Delta_2$ condition, that is, $t \Phi'(t) \le \kappa \Phi(t)$ for all $t \in \R$, cf. Lemma \ref{lem:cond2}. Then, there exists $K > 0$ such that $\Phi(t) \le K |t|^\kappa$ for all $|t| \ge r$. Observe that this and (N3) yield $\kappa > 2$. There holds
\begin{equation*}
\begin{split}
\wt J(v) & \ge c_\varepsilon \|v\|^2 - C_\varepsilon \left( \int_{\{|v| < r\}} \Phi(|v|) \, dx + \int_{\{|v| \ge r\}} \Phi(|v|) \, dx \right)\\
& \ge c_\varepsilon \|v\|^2 - \delta C_\varepsilon |v|_2^2 - K C_\varepsilon |v|_\kappa^\kappa \ge (c_\varepsilon - \delta C_\varepsilon) \|v\|^2 - K C_\varepsilon |v|_\kappa^\kappa.
\end{split}
\end{equation*}
Letting
\begin{equation*}
\beta_n := \sup \left\{ |v|_\kappa : v \in Z_n \text{ and } \|v\| = 1 \right\}
\end{equation*}
we obtain, for all $v \in Z_n$,
\begin{equation*}
\wt J(v) \ge (c_\varepsilon - \delta C_\varepsilon) \|v\|^2 - K C_\varepsilon \beta_n^\kappa \|v\|^\kappa.
\end{equation*}
Consequently, choosing
\begin{equation*}
r_n := \left( \frac{c_\varepsilon - \delta C_\varepsilon}{2 K C_\varepsilon \beta_n^\kappa} \right)^{1/(\kappa-2)}
\end{equation*}
we obtain, for all $v \in Z_n$ such that $\|v\| = 1$,
\begin{equation*}
\wt J(v) \ge \frac{c_\varepsilon - \delta C_\varepsilon}{2} r_n^2.
\end{equation*}
Arguing as in \cite[Lemma 3.8]{Willem} and using the compact embedding $H^1_{\SO(2)}(\R^2)^6 \hookrightarrow \hookrightarrow L^\kappa(\R^2)^6$, where
\[
H^1_{\SO(2)}(\R^2)^6 := \left\{ u \in H^1(\R^2)^6 : g \star u = u \text{ for all } g \in \SO(2) \right\},
\]
we get that $\lim_n r_n = \infty$.\\
Finally, (ii) is a consequence of Lemma \ref{lem:neg}.
\end{altproof}

\begin{altproof}{\eqref{eq:EM_Energy}}
Let $u$ be a critical point of $J$ and $E$ is of the form \eqref{eq:travel_wave}.
Observe that
\begin{eqnarray*}
	\langle E,D\rangle&=&	-\eps(x)\omega^2|E|^2+\chi(\langle |E|^2\rangle)|E|^2\\
	&=&\big(-V(x)+\chi(\frac12 |u|^2)\big)\big(|U|^2\cos^2(kx_3+\omega t)+|\wt U|^2\sin^2(kx_3+\omega t)\big),\\
	&\leq &\big(-V(x)+\chi(\frac12 |u|^2)\big)|u|^2,\\
	\langle B,H\rangle &=& |B|^2 = \frac{1}{\omega^2}\big|\curlop \big(U(x) \sin(kx_3+\omega t)- \wt U(x)\cos(kx_3+\omega t)\big)\big|^2\\
	&\leq & \frac{1}{\omega^2}|\cnabla\times u|^2.
\end{eqnarray*}
Therefore
$$
\cL(t)=\frac12\int_{\R^2}\int_{a}^{a+1}\langle E,D\rangle +\langle B,H\rangle\,dx_3\, d(x_1, x_2)
\leq  \frac{1}{2}\bigl(1+\frac{1}{\omega^2}\bigr) |\cnabla\times u|^2<\infty,
$$
since $J(u)<\infty$ and $J'(u)(u)=0$.
\end{altproof}

\begin{Rem}
One may consider also  the action $-\cS \colon X_{\SO(2)} \to X_{\SO(2)}$ given by
\begin{equation*}
\cS u := -u_\rho + u_\tau - u_\zeta
\end{equation*}
and note that $-\cS$ in an isometry and $J|_{X_{\SO(2)}}$ is also invariant under $-\cS$. Then
$$\big(X_{\SO(2)}\big)^{(-\cS)}:=\big\{u\in X_{\SO(2)}:-\cS u =u\big\}=\big\{u\in X_{\SO(2)}: u=u_\tau\big\}$$
and in view of  the Palais principle of symmetric criticality, critical points of $J|_{(X_{\SO(2)})^{(-\cS)}}$ are critical points of the free functional $J$. Arguing as in Section \ref{sec:Cerami} we can find infinitely many solutions of the form $u_n= \begin{pmatrix}
U_n\\
\wt U_n
\end{pmatrix}$
with
\begin{equation*}%\label{eq:shapeofsol2}
U_n = \frac{\beta_n(x)}{|x|} \begin{pmatrix}
-x_2\\
x_1\\
0
\end{pmatrix},\quad
\wt U_n = 0
\end{equation*}
for some radial $\beta_n:\R^2\to\R$.  Hence $E=U_n\cos(kx_3+\omega t)$ is a TE-mode. It is not clear however, if for the Kerr nonlinearity  $\beta_n$ coincide with the solutions obtained in \cite{McLeodStuartTroy} by means of ODE methods.
\end{Rem}

\vspace{8mm}
{\bf Acknowledgments.} We thank the anonymous reviewer their valuable comments, which helped us improve our manuscript.
\vspace{2mm}

{\bf Conflict of interest.} All authors declare that they have no conflicts of interest.
\vspace{2mm}

{\bf Ethics approval.} Ethical approval was not sought for the present study.
\vspace{2mm}

{\bf Funding.} The authors were partly supported by the National Science Centre, Poland (Grant No. 2020/37/B/ST1/02742). J.S. is a member of GNAMPA (INdAM) and is supported by the GNAMPA project {\em Metodi variazionali e topologici per alcune equazioni di Schr\"odinger nonlineari}.

%
%{\bf Compliance with Ethical Standards.} The authors declare that they have no conflict of interests, they also confirm that the manuscript complies to the Ethical Rules applicable for this journal.
%
\vspace{2mm}

{\bf Data availability statement.}  Data sharing not applicable to this article as no datasets were generated or analysed during the current study.

\end{document}